%% file: Ptolemy-arxiv-ver3.tex
\documentclass[11pt]{article}

\usepackage{amsmath}
\usepackage{authblk}

\usepackage{fullpage}
\usepackage{hyperref}
\usepackage{amssymb}
\usepackage{url}

\usepackage[main=english,russian]{babel}

\usepackage{tikz}
\usetikzlibrary{math}

\usepackage{basic-commands}

\newenvironment{acknowledgment}{%
	\begin{abstract}
	}{%
	\end{abstract}
}

\title{The Four Point Condition: An Elementary Tropicalization of Ptolemy's Inequality}

\author[1]{Mario G\'omez}
\author[2]{Facundo M\'emoli}

\affil[1]{Department of Mathematics, 
	 	The Ohio State  University.\\
	 	\texttt{gomezflores.1@osu.edu}}

\affil[2]{Department of Mathematics and Department of Computer Science and Engineering, 
	 	The Ohio State University.\\ %
	 	\texttt{memoli@math.osu.edu}}

\date{\today}

\begin{document}
	
\maketitle
\begin{abstract}
	Ptolemy's inequality is a classic relationship between the distances among four points in Euclidean space. Another relationship between six distances is the 4-point condition, an inequality satisfied by the lengths of the six paths that join any four points of a metric (or weighted) tree. The 4-point condition also characterizes when a finite metric space can be embedded in such a tree. The curious observer might realize that these inequalities have similar forms: if one replaces addition and multiplication in Ptolemy's inequality with maximum and addition, respectively, one obtains the 4-point condition. We show that this similarity is more than a coincidence. We identify a family of Ptolemaic inequalities in CAT-spaces parametrized by a real number and show that a certain limit involving these inequalities, as the parameter goes to negative infinity, yields the 4-point condition, giving an elementary proof that the latter is the tropicalization of Ptolemy's inequality.
\end{abstract}

\paragraph{Ptolemy's inequality.}
Ptolemy's inequality is a classical result in Euclidean geometry that relates the six distances among four points in the plane. Given $p_1,p_2,p_3,p_4 \in \mathbb{R}^2$ with $d_{ij} := \|p_i-p_j\|$, the result states
\begin{equation}\label{ineq:ptolemy}
	d_{13}\,d_{24} \leq d_{12}\,d_{34}+d_{23}\,d_{41}.
\end{equation}
\begin{figure}[ht]
	\begin{minipage}{0.48\linewidth}
		\centering
		\input{figures/Quadrilateral.tex}
	\end{minipage}
	\hfill
	\begin{minipage}{0.48\linewidth}
		\centering
		\input{figures/Ptolemy-3D.tex}
	\end{minipage}
	\caption{Left: A quadrilateral on the plane. Right: A visualization of Ptolemy's inequality in $\mathbb{R}^n$. It suffices to consider $\mathbb{R}^3$ since any 4 points in general position generate a 3-dimensional subspace. With $p_1,p_2,p_4$ fixed, move $p_3$ along the circle where the distances from $p_3$ to $p_2$ and $p_4$ remain fixed until all four points are coplanar. There are two choices for the ending point $p_3'$, but if $p_1,p_2,p_4$ are not collinear, we can make the choice where $d(p_1,p_3) \leq d(p_1,p_3')$. Then the inequality on $\{p_1,p_2,p_3',p_4\}$ implies inequality (\ref{ineq:ptolemy}) because the only distance that changed was $d(p_1,p_3)$ to $d(p_1,p_3')$.}
	\label{fig:quadrilateral}
\end{figure}
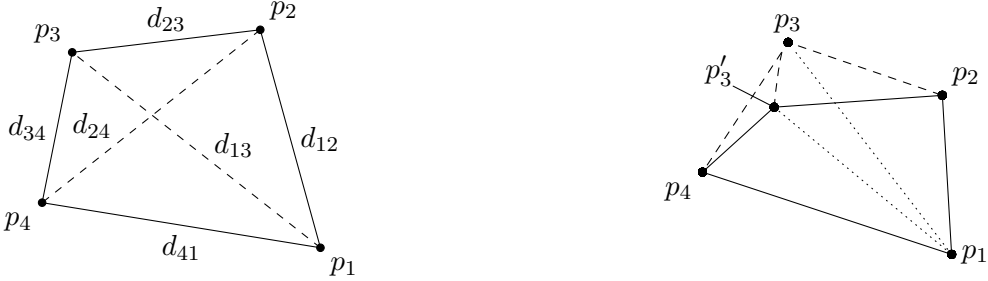
The case of equality, known as Ptolemy's Theorem, happens if and only if the points lie on a circle or a line. In fact, the inequality holds in every $\mathbb{R}^n$; see Figure \ref{fig:quadrilateral}.\\
\indent Inequality (\ref{ineq:ptolemy}) was discovered by Claudius Ptolemy, a Greek mathematician and astronomer who lived approximately between 100 and 170 AD. His treatise, the \textit{Almagest}, contains one of the most thorough tables of chords\footnote{A chord is the line segment joining two points on a circle whose length is given by the formula $R \cdot \sin(\alpha/2)$, where $R$ is the radius of the circle and $\alpha$ is the angle subtended by the chord.} of his time. Much like $\sin(x)$ and $\cos(x)$ are today, \emph{chords} were the primary trigonometric function in ancient Greece. Ptolemy computed the length of the chords on a circle of diameter 120 for angles between $0.5^\circ$ and $180^\circ$ in increments of half a degree. Among other tools, including an approximation similar to $\sin(x) \approx x$ for chords of small angles, he developed formulas for the chord of a half-angle and a double angle using the equality in equation (\ref{ineq:ptolemy}). The contemporary equivalent of these formulas is the identity for the sine of a sum of angles which, perhaps unsurprisingly, follows from Ptolemy's theorem. The \textit{Almagest} continued to be an authoritative text for centuries, not only for its mathematical results, but also for its astronomical models, the geocentric view of the universe being one example. The precise computations required for these and subsequent cosmological studies were made possible thanks to the precision in Ptolemy's table of chords.
\paragraph{Trees and the 4-point condition.}
In a completely different context, a \emph{graph} $G$ is defined as a pair $(V,E)$, where an element $v \in V$ is called a vertex (or node) and $e \in E$ is called an edge. Edges $e \in E$ connect two distinct vertices\footnote{Notice that our definition of edges does not allow loops nor multiple edges. These types of graphs are called simple.} $v_1, v_2 \in V$ and we denote this by $e = \{v_1,v_2\}$. A path between two vertices $v_1,v_2$ is a sequence of edges $e_i = \{w_i,w_i'\}$ for $i=1,\dots,n$, where $w_1=v_1$, $w_{i+1} = w_i'$, and $w_n'=v_2$. If we assign a weight $\ell_e > 0$ to every edge $e \in E$, then the path $e_1, \dots, e_n$ has length $\ell_{e_1} + \cdots + \ell_{e_n}$. We define the distance between two vertices $v_1, v_2$ to be the length of the shortest path joining them.\\
\indent \emph{Phylogenetics}, the study of evolutionary relations between species, is one area that benefits from assigning distances to graphs. The main object in phylogenetics is the \emph{evolutionary} or \emph{phylogenetic tree}. This is a graph with no cycles where each point represents a species and an edge connects two species if one evolved from the other. In particular, whenever there is a mutation on a species at a node $v$, a new vertex appears for the mutated organism with an edge linking it to $v$. Some phylogenetic trees are \emph{rooted}, which means that every species represented in the tree has a common evolutionary ancestor at a node $v_0$; see Figure \ref{fig:Haeckel_tree}. Additionally, by measuring quantities such as the frequency of alleles in a population (that is, gene variations), it is possible to define a distance on this tree quantifies the genetic difference between two species. In that regard, a common problem in this field is finding a good tree representation for a set of species. Mathematically, this means that given a finite set $X$ with a distance function $d_X$, we want to determine whether there exists a tree $T$ and an isometric embedding\footnote{An isometric embedding is a function $f:(X,d_X) \to (Y,d_Y)$ between metric spaces such that $d_Y(f(x), f(x')) = d_X(x,x')$ for all $x,x' \in X$.} $\iota:X \to T$ such that $\iota(x)$ is a vertex of $T$ for all $x \in X$. The answer is given by the following inequality involving 4-tuples of points. Let $x_1,x_2,x_3,x_4 \in X$ and denote their distances by $d_{ij} := d_X(x_i,x_j)$. Such a tree $T$ exists if and only if the inequality
\begin{equation}\label{ineq:4pt_condition}
	d_{13}+d_{24} \leq \max(d_{12}+d_{34}, d_{23}+d_{41})
\end{equation}
holds for every set of four points in $X$. Inequality (\ref{ineq:4pt_condition}) is known as the \emph{4-point condition} and such a metric space $X$ is called \emph{tree-like}. This result was proved by Zaretskii in 1965 \cite{4pc-zar65} in the case of integer distances and extended by Sim\~oes Pereira for non-integer distances in 1969 \cite{4pc-simoes-pereira}. It was also proved independently by Buneman in the 1970's \cite{4pc-buneman-1971, 4pc-buneman-1974}. There is a generalization of this result where $X$ is not required to be finite. In that case, $X$ might not be embeddable in a tree but instead in a more general object known as an \emph{$\mathbb{R}$-tree}. We give the definition of $\mathbb{R}$-trees in Section \ref{sec:definitions}.\\
\begin{figure}[ht]
	\centering
	\includegraphics[scale=0.55]{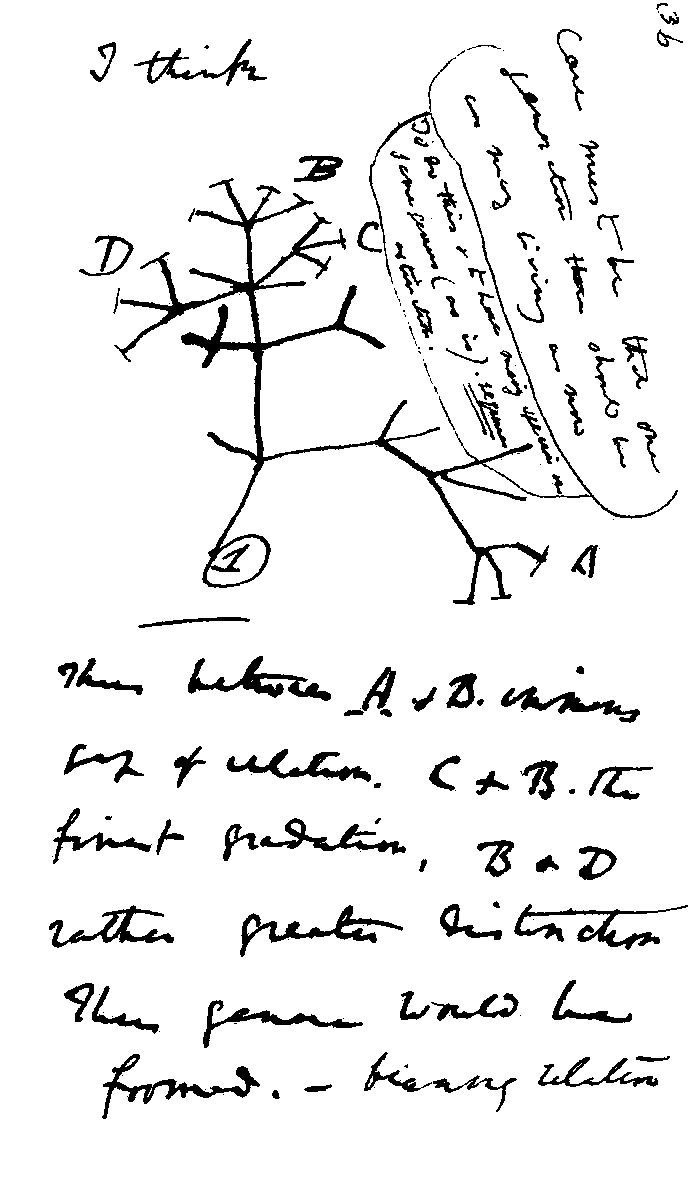}
	\caption{An excerpt of Darwin's notebook with his first phylogenetic tree. Image from \url{https://commons.wikimedia.org/wiki/File:Darwin_tree.png}}
	\label{fig:Darwin_tree}
\end{figure}
\begin{figure}[ht]
	\centering
	\includegraphics[scale=0.1]{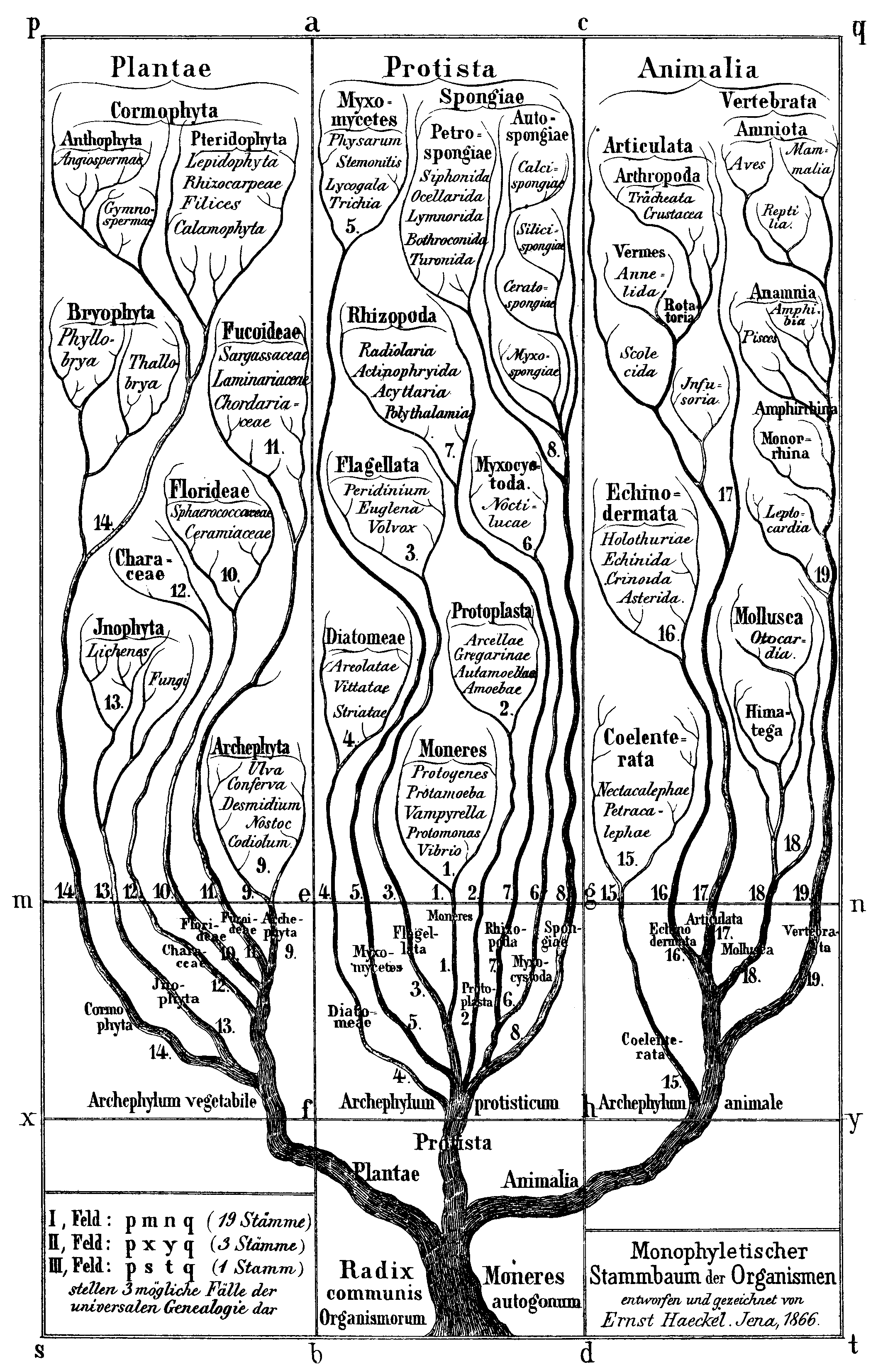}
	\caption{A tree of life by Haeckel with three kingdoms: Plantae, Protista, and Animalia. Image from \url{https://en.wikipedia.org/wiki/File:Haeckel_arbol_bn.png}}
	\label{fig:Haeckel_tree}
\end{figure}
\paragraph{Ptolemy's inequality and the 4-point condition.}
The reader might notice that the difference between inequalities (\ref{ineq:ptolemy}) and (\ref{ineq:4pt_condition}) is that the product of two numbers in the former is replaced with addition in the latter; similarly, addition in (\ref{ineq:ptolemy}) is replaced with a maximum in (\ref{ineq:4pt_condition}). In other words, the four-point condition is a \emph{tropical} version of Ptolemy's inequality. It turns out that this is more than a visual similarity. Tropical Algebraic Geometry is the study of polynomials where the usual addition ($+$) and multiplication ($\times$) are replaced\footnote{Some sources replace $+$ with $\oplus = \min$, but the theories induced by $(\min,+)$ and $(\max,+)$ are equivalent. We use $(\max,+)$ to simplify our exposition.} with $x \oplus y = \max(x,y)$ and $x \otimes y = x + y$ \cite[Section 1.1]{intro-tropical-geometry}. In classical Algebraic Geometry, Ptolemy's theorem $d_{13} d_{24} = d_{12}d_{34} + d_{41}d_{23}$ is the \emph{Pl\"ucker relation} that characterizes the Grassmannian $G_{2,4}$; see Theorem 3.20 and the explanation after Corollary 3.21 in \cite{algebraic-statistics-biology}. On the other hand, the set $\mathcal{T}_n$ of tree-like metric spaces with $n$ points is the tropical Grassmannian $\mathcal{T}_{2,4}$ \cite[Theorem 3.45]{algebraic-statistics-biology}, and its Pl\"ucker relation is $d_{13} + d_{24} \leq \max(d_{12}+d_{34}, d_{41}+d_{23})$, that is, the four-point condition. These ideas, including the tropicalization of the Pl\"ucker relations, are discussed in Section 3.5 of \cite{algebraic-statistics-biology}. In particular, see Example 3.46 and the first paragraph of page 128. In this paper, we take a more geometric approach and use a certain generalized notion of \emph{curvature} to build a link between Ptolemy's inequality and the four-point condition.

\paragraph{Curvature and generalizations of Ptolemy's inequality.}
There exist extensions of Ptolemy's results in Euclidean, spherical, and hyperbolic geometries. In 1947, Haantjes \cite{ptolemy-haantjes} proved a version of Ptolemy's inequality in the spherical and hyperbolic planes. Proofs that equality holds for cyclic quadrilaterals were given by Shirokov \cite{ptolemy-shirokov} in 1924 and Perron \cite{ptolemy-perron} in 1964 for the case of hyperbolic geometry, and by Zeitler \cite{ptolemy-zeitler} in 1966 for non-Euclidean geometries. The work of Kurnik and Volenic \cite{ptolemy-glasnik} in 1967 contains several related results, including an extension of Ptolemy's theorem to $n$ points in any of the three plane geometries using a matrix similar to the $P_\kappa(p_1,\dots,p_m)$ that we consider in Section \ref{sec:val-proofs}.\\
\indent These papers, however, do not consider the converse: what can we say about any four points that satisfy the equality in the non-Euclidean Ptolemaic theorem? This question was addressed in 1970 by Joseph E. Valentine \cite{ptolemy-hyperbolic, ptolemy-spherical}. He gave a proof of the non-Euclidean Ptolemy's inequality and its analogue: if equality holds in the spherical inequality, the four points must lie on a circle \cite[Theorem 5.4]{ptolemy-spherical}; in the hyperbolic case, equality implies that the four points lie on either a circle, a line, a horocycle, or one branch of an equidistant curve \cite[p. 1]{ptolemy-hyperbolic}. Andalafte and Valentine further generalized both the theorem and the inequality to higher dimensional hyperbolic space in 1971. In this case, $n+2$ points in $n$-dimensional hyperbolic space lie on an $(n-1)$-dimensional subspace, an $(n-1)$-dimensional sphere or limiting surface, or a sheet of an $(n-1)$-dimensional equidistant surface if and only if the determinant we call $\gamma_{-1}(p_1, \dots, p_{n+2})$ vanishes \cite[Theorem 4.7]{ptolemy-higher-dim}.\\
\begin{figure}
	\centering
	\includegraphics[scale=0.45]{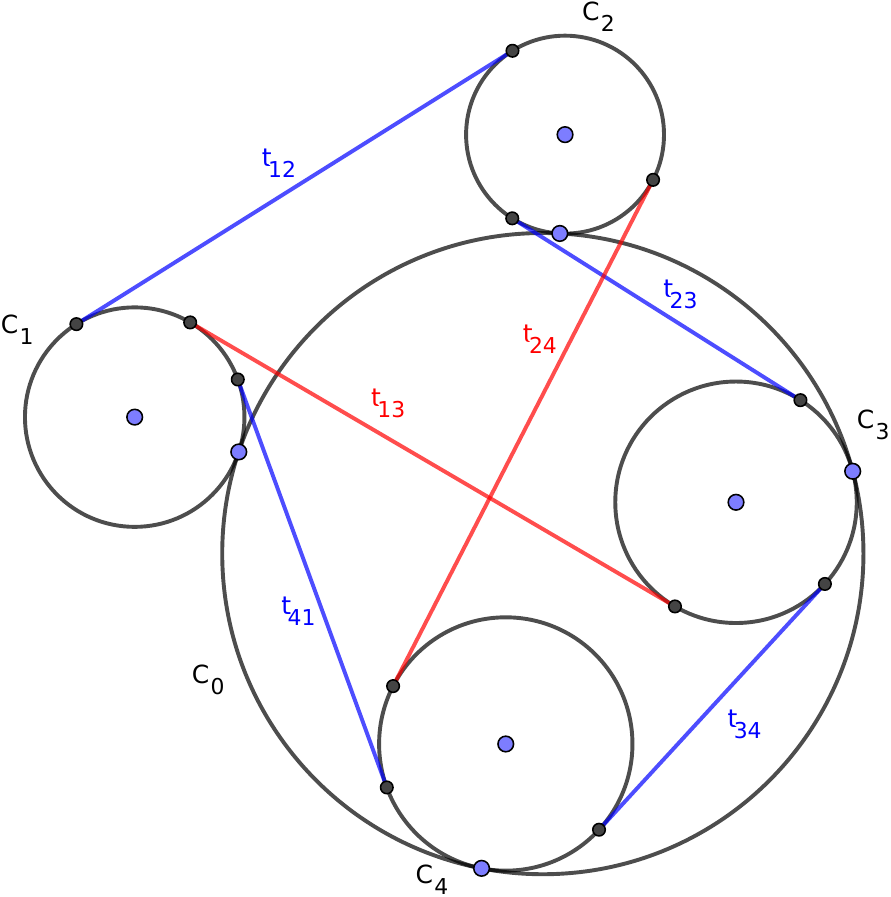}
	\caption{Casey's theorem. Notice that $C_1$ and $C_2$ are in the exterior of $C_0$ while $C_3$ and $C_4$ are in the interior. For this reason, the tangent line between $C_1$ and $C_2$ and that between $C_3$ and $C_4$ are outer tangents, and every other segment is an inner tangent.}
\end{figure}
\indent The Euclidean Ptolemaic theorem has an extension known as Casey's theorem. Let $C_1, C_2, C_3, C_4$ be four non-intersecting circles tangent to a fifth circle $C_0$ in that order. If $C_i$ and $C_j$ are both in the interior or both in the exterior of $C_0$, let $t_{ij}$ be the length of a common outer tangent to $C_i$, $C_j$. Otherwise, let $t_{ij}$ be the length of a common inner tangent\footnote{There are four common tangents between any two non-intersecting circles such that neither is in the interior of the other. Such a tangent is called \emph{inner} if it intersects the line that joins the centers of the circles and \emph{outer} otherwise. Each such pair of circles has two inner and two outer tangents, and both inner (resp. outer) tangents have the same length.}. Then, the lengths $t_{ij}$ satisfy the Ptolemaic identity
\begin{equation}\label{eq:Casey_theorem}
	t_{13}\,t_{24} = t_{12}\,t_{34}+t_{23}\,t_{41}.
\end{equation}
Conversely, for any four circles $C_1, C_2, C_3, C_4$ and choices of inner or outer tangents, if the lengths $t_{ij}$ satisfy (\ref{eq:Casey_theorem}), then there exists a circle $C_0$ that is tangent to all $C_i$. The degenerate case where all circles have radius 0 is Ptolemy's theorem.\\
\indent This is another result that has non-Euclidean generalizations. Casey proved that (\ref{eq:Casey_theorem}) holds in 1866 \cite[Article 1]{casey-theorem} (see also \cite[p. 103]{sequel-euclid}). The proof of the converse appears in a book by Johnson published in 1929 \cite[p. 121]{modern-geometry-johnson}. The spherical case appears in the book by M'Clelland and Preston from 1886 \cite{treatise-spherical-trig} and the hyperbolic case, in an article by Kubota from 1912 \cite{casey-theorem-kubota}. Kurnick and Volenic gave a version of Casey's theorem for $n$ circles tangent to a common circle in all three geometries \cite[Theorem 9]{ptolemy-glasnik}. More recently, in 2015 Abromsimov and Mikaiylova \cite{casey-theorem-non-euclidean-planes} gave another proof for both non-Euclidean geometries (although the details are written specifically for the hyperbolic case). Abrosimov and Aseev extended the Euclidean theorem to higher dimensions \cite{casey-theorem-higher-dimensions} in 2018.\\
\indent The hyperbolic \cite{ptolemy-hyperbolic} and spherical \cite{ptolemy-spherical} planes are the prototypical examples of geometries with constant negative and positive curvature. We take this a step further and generalize these inequalities to a class of geodesic metric spaces known as \emph{CAT-spaces}. For now, let's say that a $\text{CAT}(\kappa)$ space is a metric space with curvature bounded above by $\kappa$ (see Section \ref{sec:definitions} for a precise definition). In particular, a $\text{CAT}(0)$ space has non-positive curvature. Furthermore, $\text{CAT}(\kappa)$ spaces with $\kappa<0$ are known to satisfy a property known as hyperbolicity, which can be understood as a relaxation of the 4-point condition (\ref{ineq:4pt_condition}). Indeed, if there exists $\delta \geq 0$ such that for all $x_1,x_2,x_3,x_4 \in X$, the inequality
\begin{equation}\label{ineq:hyperbolicity}
	d_{13}+d_{24} \leq \max(d_{12}+d_{34}, d_{23}+d_{41}) + 2\delta
\end{equation}
is satisfied, then $X$ is called \emph{$\delta$-hyperbolic}\footnote{We use a rearrangement of the original inequality in \cite[p. 89]{gromov-hyperbolic}. See also the discussion after Definition 1.20, Chapter III.H, of \cite{non-positive-curvature}}. The \emph{hyperbolicity} of $X$ is defined as the infimum $\delta_0$ of all $\delta$ that satisfy inequality (\ref{ineq:hyperbolicity}), and we denote it by $\operatorname{hyp}(X) := \delta_0$. $\text{CAT}(\kappa)$ and $\delta$-hyperbolic spaces are, in general, different. For $\kappa<0$, every $\text{CAT}(\kappa)$ space $X$ has $\operatorname{hyp}(X) \leq \ln(2)/\sqrt{-\kappa}$. However, not all $\delta$-hyperbolic spaces for a fixed $\delta \geq 0$ are $\text{CAT}(0)$. For example, every compact metric space is $\delta$-hyperbolic for $\delta \geq \mathbf{diam}(X)$. In particular, the unit sphere $\mathbb{S}^{2}$ with geodesic distance is $\pi$-hyperbolic but it is not $\text{CAT}(0)$ since its curvature is strictly positive; see do Carmo's book \cite[Chapter 3, Example 7]{do-carmo} or our Remark in Section \ref{sec:definitions}.
\paragraph{Hyperbolic groups.}
\indent The hyperbolicity of a metric space first appeared in the work of Gromov \cite[p. 89]{gromov-hyperbolic}. Gromov wanted to define a notion of hyperbolic group with the objective of translating the rich geometry of hyperbolic spaces into group-theoretic results. We describe one of several equivalent constructions. Let $G = \langle S \rangle$ be a group generated by a finite set $S = \{a_1, a_2, \dots, a_n\} \subset G$. Consider a graph $\Gamma(G,S)$ with vertex set $G$, and draw an edge between vertices $g_1$ and $g_2$ if there exists $a_i \in S$ such that $g_2 = a_i \cdot g_1$. $\Gamma(G,S)$ is then made into a metric space by assigning a weight of 1 to each edge. $\Gamma(G,S)$ is known as the \emph{Cayley graph} of $G$ with respect to $S$. A Cayley graph induces a metric on $G$, known as a \emph{word metric}, by setting the distance between $g_1$ and $g_2$ as the length of the shortest path between $g_1$ and $g_2$ in $\Gamma(G,S)$. At this point, one might be tempted to \textit{define} a group $G$ to be hyperbolic if its word metric generated by a Cayley graph $\Gamma(G,S)$ is $\delta$-hyperbolic for some $\delta>0$. However, the construction of $\Gamma(G,S)$ depends on the generating set and, in general, two Cayley graphs $\Gamma(G,S_1), \Gamma(G,S_2)$ for the same group might not be isometric. Despite this, the above notion of a hyperbolic group is well-defined under \emph{quasi-isometry}. A function $f:(X,d_X) \to (Y,d_Y)$ between metric spaces is called a quasi-isometry if there exist constants $A \geq 1$ and $B, C\geq 0$ such that for all $x_1,x_2 \in X$,
\begin{equation*}
	\frac{1}{A} d_X(x_1,x_2) - B \leq d_Y(f(x_1), f(x_2)) \leq A d_X(x_1,x_2) + B,
\end{equation*}
and if for all $y \in Y$, there exists $x_0 \in X$ such that $d_Y(y,f(x_0)) \leq C$. In other words, $f$ is a quasi-isometry if it is almost an isometry and almost surjective as measured by the inequalities above. Returning to our discussion of groups, one of the basic results in geometric group theory says that every pair of Cayley graphs $\Gamma(G,S_1)$ and $\Gamma(G,S_2)$ of the same group are quasi-isometric, irrespective of the generating sets $S_1$ and $S_2$. Furthermore, hyperbolicity is a quasi-isometry invariant. The constants $\delta$ corresponding to different quasi-isometric spaces are usually distinct, but the point is that they all exist. Thus, defining a group to be \emph{hyperbolic} if any of its word metrics is $\delta$-hyperbolic (for some $\delta \geq 0$) is an unambiguous notion.\\
\begin{figure}
	\centering
	\includegraphics[scale=0.20]{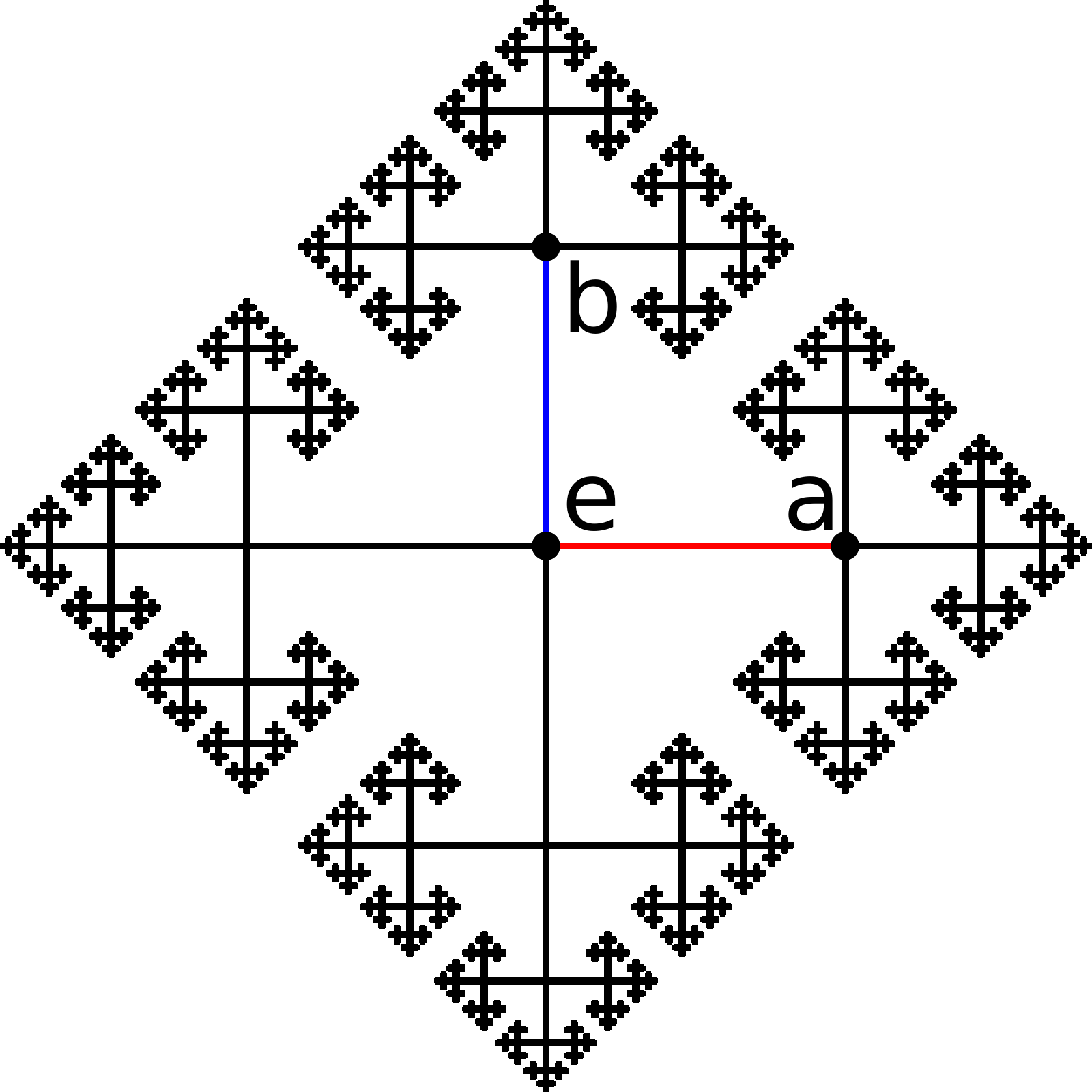}
	\caption{The Cayley graph $\Gamma(F_2, S)$ of the free group $F_2 = \langle S \rangle$ generated by two elements $S = \{ a,b \}$. This graph is a tree and $F_2 \subset \Gamma(F_2,S)$ is a tree-like metric space. Image from \url{https://upload.wikimedia.org/wikipedia/commons/d/d2/Cayley_graph_of_F2.svg}}
	\label{fig:cayley_graph}
\end{figure}
\indent Hyperbolic groups enjoy several properties. For example, the \emph{word problem} on a given finitely generated group $G = \langle S \rangle$ amounts to the question of constructing an algorithm that can determine whether any element $g \in G$ is the identity $e \in G$ using only the properties of $S$. While the word problem is not solvable in general, hyperbolic groups are one class where not only such an algorithm exists, but it runs in quadratic time in the length\footnote{The length of an element (or word) $g = s_1 \cdots s_m \in \langle S \rangle$ is $m$.} of the word; see Theorems 3.4.5 and 2.3.10 of \cite{word-processing-groups}. Hyperbolic groups are also finitely presented and act properly discontinuously and cocompactly by isometries\footnote{The action of a group $G$ on a locally compact $X$ is properly discontinuous if for every compact set $K \subset X$, there are only a finite number of elements $g \in G$ for which $(g\cdot K)\cap K$ is non-empty. An action is cocompact if $X/G$ is compact. We say that $G$ acts by isometries if the function $x \mapsto g \cdot x$ is an isometry for every $g \in G$.} on $\delta$-hyperbolic spaces. This is another setting where the notions of $\text{CAT}$ spaces and $\delta$-hyperbolicity differ. We can make an analogous definition and say that a $\text{CAT}(\kappa)$ group is one that has a properly discontinuous and cocompact action by isometries on a $\text{CAT}(\kappa)$ space. Since $\text{CAT}(\kappa)$ spaces are $\delta$-hyperbolic when $\kappa<0$, every $\text{CAT}(\kappa)$ group is hyperbolic. The converse, however, is an important unsolved question in geometric group theory (see Remark 2.3 in Chapter III.$\Gamma$ of \cite{non-positive-curvature}).\\
\indent Going back to metric trees, their relationship with hyperbolicity goes beyond satisfying inequality (\ref{ineq:hyperbolicity}) for $\delta=0$. One of Gromov's definitions of hyperbolic groups says that, intuitively, a group whose ``word metric looks like a tree when observed from infinity'' is hyperbolic \cite[p. 78]{gromov-hyperbolic}. This motivates thinking of trees and their properties as the limit at infinity of the corresponding properties of a certain family of spaces. Indeed, metric trees (and, more generally, $\mathbb{R}$-trees) are $\text{CAT}(\kappa)$ spaces for all $\kappa \in \mathbb{R}$. This raises the question of whether the 4-point condition (\ref{ineq:4pt_condition}) might be the $(-\infty)$-version of a certain family of inequalities. It turns out that a generalization of Ptolemy's inequality in CAT-spaces gives a family which realizes the idea.

\paragraph{Contributions.}
\indent For the sake of completeness, in this paper we adapt the proof of the non-Euclidean Ptolemaic inequalities given by Valentine in his 1970 papers \cite{ptolemy-hyperbolic, ptolemy-spherical}. This proof is valid in any dimension, similar to \cite{ptolemy-higher-dim}, and works simultaneously for positive and negative curvature. We then show that $\text{CAT}(\kappa)$ spaces also satisfy a $\kappa$-Ptolemaic inequality, which is either the spherical or hyperbolic inequality depending on whether $\kappa>0$ or $\kappa<0$. Lastly, we show that the four-point condition is indeed the limit of the $\kappa$-Ptolemy inequality when $\kappa \to -\infty$, essentially characterizing the former as the $(-\infty)$ form of the latter.

\section{Definitions.}\label{sec:definitions}
Let $x, y \in \mathbb{R}^{n+1}$. Define
\begin{align*}
	\langle x,y \rangle &:= \phantom{-}x_0y_0 + \dots + x_ny_n,\\
	\langle x|y \rangle &:= -x_0y_0 + \dots + x_ny_n.
\end{align*}
These are the usual inner product in $\mathbb{R}^{n+1}$ and the Minkowski bilinear form, respectively. The $n$-dimensional space with constant curvature $\kappa$ is defined as:
\begin{align*}
	M_\kappa^n &:= \left\{x \in \mathbb{R}^{n+1} \ | \ \langle x,x \rangle = \frac{1}{\kappa} \right\}, \text{ if } \kappa>0,\\
	M_0^n &:= \mathbb{R}^n, \text{ if } \kappa=0, \text{ and}\\
	M_\kappa^n &:= \left\{x \in \mathbb{R}^{n+1} \ | \ \langle x|x \rangle = \frac{1}{\kappa} \text{ and } x_0>0 \right\}, \text{ if } \kappa<0.
\end{align*}

These spaces come equipped with a geodesic distance given by:
\begin{equation*}\label{eq:hyperbolic_distance}
	d_{M_\kappa^n}(x,y) :=
	\begin{cases}
		\frac{1}{\sqrt{+\kappa}} \arccos(\kappa \langle x,y \rangle), & \text{if } \kappa>0,\\
		\|x-y\|, & \text{if } \kappa=0,\\
		\frac{1}{\sqrt{-\kappa}} \operatorname{arcosh}(\kappa \langle x|y \rangle), & \text{if } \kappa<0.
	\end{cases}
\end{equation*}
\indent In other words, if $\kappa>0$, $M_{\kappa}^{n}$ is the $n$-dimensional sphere $\mathbb{S}^{n} \subset \mathbb{R}^{n+1}$ with radius $1/\sqrt{\kappa}$. If $\kappa<0$, $M_{\kappa}^{n}$ is a rescaling of the $n$-dimensional hyperbolic space $\mathbb{H}^{n}$ by a factor of $1/\sqrt{-\kappa}$. Notice that the diameter $D_\kappa$ is $\pi/\sqrt{\kappa}$ if $\kappa>0$ and $\infty$ when $\kappa \leq 0$.\\

\indent More generally, we can assign a notion of curvature to a geodesic metric space $(X, d_X)$ by comparing it with the model spaces defined above. We adopt the definitions from Bridson and Haefliger \cite[Chapter II.1]{non-positive-curvature}. Given two points $x,y \in X$, let $[x,y]$ denote any geodesic joining $x$ and $y$. For a set of three points $x,y,z \in X$, consider the geodesic triangle $\Delta$ formed by the three segments $[x,y]$, $[y,z]$, $[z,x]$. If $d_X(x,y)+d_X(y,z)+d_X(z,x) < 2D_\kappa$, there exists a geodesic triangle $\overline{\Delta} \subset M_\kappa^2$ with vertices $\overline{x}, \overline{y}, \overline{z}$, such that $d_X(x,y) = d_{M_\kappa^2}(\overline{x}, \overline{y})$, $d_X(y,z) = d_{M_\kappa^2}(\overline{y}, \overline{z})$, and $d_X(z,x) = d_{M_\kappa^2}(\overline{z}, \overline{x})$. $\overline{\Delta}$ is called a comparison triangle for $\Delta$, and a point $\overline{p} \in [\overline{x}, \overline{y}]$ is called a comparison point for $p \in [x,y]$ if $d_X(p,x) = d_{M_\kappa^2}(\overline{p}, \overline{x})$ (similarly for the other two segments). $X$ will be called a $\text{CAT}(\kappa)$ space if for every geodesic triangle $\Delta$ and any two points $p,q \in \Delta$, their comparison points satisfy
\begin{equation}\label{ineq:cat_inequality}
	d_X(p,q) \leq d_{M_\kappa^2}(\overline{p}, \overline{q}).
\end{equation}
\indent We can summarize the above construction as follows. See also Figure \ref{fig:comparison_triangles}. Given a geodesic triangle $\Delta \subset X$ (with bounded perimeter if $\kappa>0$), we can find a comparison triangle $\overline{\Delta} \subset M_\kappa^2$ whose sides have the same length as those of $\Delta$. $X$ is a $\text{CAT}(\kappa)$ space if the triangle $\Delta$ is thinner than $\overline{\Delta}$ (as specified by inequality (\ref{ineq:cat_inequality})).

\begin{remark}
	The unit sphere $M_1^{2} = \mathbb{S}^{2}$ is not $\text{CAT}(0)$. To see this, consider the geodesic triangle with vertices $x = (0,0,1)$, $y=(0,1,0)$, and $z=(0,0,1)$. The distance between each pair of points is $\pi/2$, and $d_{\mathbb{S}^{2}}(x,y) + d_{\mathbb{S}^{2}}(y,z) + d_{\mathbb{S}^{2}}(z,x) < 2\pi$. Let $p = z$ and $q = (\frac{\sqrt{2}}{2}, \frac{\sqrt{2}}{2}, 0) \in [x,y]$. For any comparison triangle $\overline{\Delta} \subset M_0^2 = \mathbb{R}^2$, $\overline{q}$ is the leg of the median from $\overline{p}$. Since $\overline{\Delta}$ is an equilateral triangle of side length $\frac{\pi}{2}$, $d_{M_0^2}(\overline{p}, \overline{q}) = \|\overline{p}-\overline{q}\| = \frac{\sqrt{3}}{2} \frac{\pi}{2}$. However, $d_{\mathbb{S}^{2}}(p,q) = \pi/2$, so $p, q$ and their comparison points $\overline{p}, \overline{q}$ fail inequality (\ref{ineq:cat_inequality}).
\end{remark}

\begin{figure}[ht]
	\centering
	\input{figures/comparison_triangle.tex}
	
	\caption{A geodesic triangle $\Delta$ in a $\text{CAT}(\kappa)$ space $X$ with vertices $x,y,z \in X$, and a comparison triangle $\overline{\Delta}$ in the model space $M_\kappa^2$ with vertices $\overline{x}, \overline{y}, \overline{z} \in M_\kappa^2$. There are two points $p \in [x,y]$ and $q \in [x,z]$, and their corresponding comparison points $\overline{p} \in [\overline{x}, \overline{y}]$ and $\overline{q} \in [\overline{x}, \overline{z}]$ satisfy $d_X(p,q) \leq d_{M_\kappa^2}(\overline{p}, \overline{q})$.}
	\label{fig:comparison_triangles}
\end{figure}
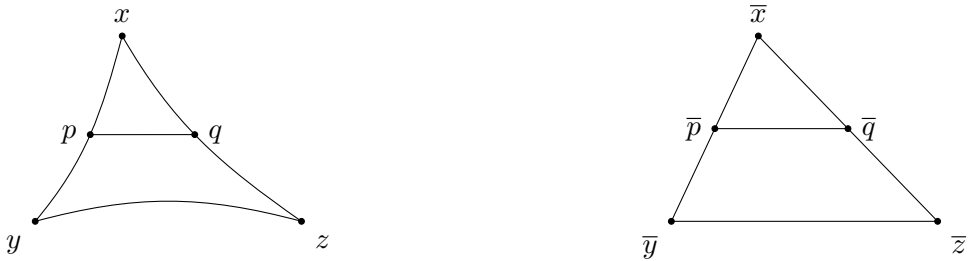

\indent Following \cite[Example II.1.15(5)]{non-positive-curvature}, an \emph{$\mathbb{R}$-tree} is a metric space $T$ such that:
\begin{enumerate}
	\item For every pair of points $x, y \in T$, there exists a unique geodesic $[x,y] \in T$ joining them.
	\item If the intersection of two geodesics $[y,x]$ and $[x,z]$ is the singleton $\{x\}$, then their union $[y,x] \cup [x,z]$ is the geodesic $[y,z]$.
\end{enumerate}

Metric trees are examples of $\mathbb{R}$-trees and, crucially, $\mathbb{R}$-trees are $\text{CAT}(\kappa)$ spaces for all $\kappa$ \cite{non-positive-curvature}. Since metric trees have no cycles, a geodesic triangle $\Delta$ can only be either a line segment or three geodesics joined at a central point which is, intuitively, an infinitely thin triangle (Figure \ref{fig:triangles_varying_curvature}). See \cite[Chapter 3]{R-trees} for more details on $\mathbb{R}$-trees.

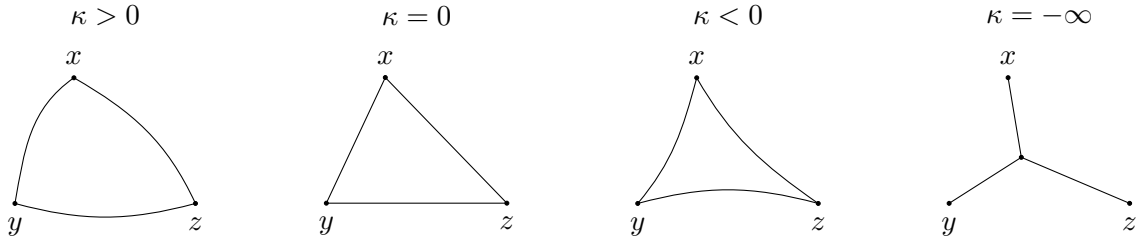
\begin{figure}[ht]
	\centering
	\input{figures/curvature_triangles.tex}
	
	\caption{Geodesic triangles in spaces with constant curvature. Moving from left to right, the curvature is positive, zero, and then negative until it reaches $-\infty$. Meanwhile, the triangles become increasingly thinner until only the union of three line segments at a common vertex remains when the curvature reaches $-\infty$.}
	\label{fig:triangles_varying_curvature}
\end{figure}

\paragraph{Notation.} Throughout the paper, whenever we have points $p_1,\dots,p_m$ in a single fixed metric space (either $X$ or $M_\kappa^n$), we denote the distance between $p_i$ and $p_j$ by $d_{ij}$. However, if $X$ is specifically a $\text{CAT}$-space, we write the points of $X$ as $p_i$ and the corresponding comparison points in the model space $M_\kappa^2$ as $\overline{p}_i$. Similarly, we use $d_{ij}$ for the distances between points in $X$ and $\overline{d}_{ij}$ for the distances in the model space $M_\kappa^2$ (instead of the more verbose notation $d_{M_\kappa^2}(\overline{p}, \overline{q})$).

\section{Generalizing Ptolemy's inequality.}\label{sec:val-proofs}
In this section, we adapt the arguments in \cite{ptolemy-higher-dim, ptolemy-hyperbolic, ptolemy-spherical} to give a unified proof of the generalized Ptolemy's inequality in every dimension for positive and negative curvature.\\
\indent Define the functions $c_\kappa: \mathbb{R}_+ \to \mathbb{R}$ and $s_\kappa:\mathbb{R}_+ \to \mathbb{R}$ as follows:
\begin{equation*}
	c_\kappa(t) :=
	\begin{cases}
		\cos\left( \sqrt{ \kappa} t \right), & \text{if } \kappa>0,\\
		1, & \text{if } \kappa=0,\\
		\cosh\left(\sqrt{-\kappa} t\right), & \text{if } \kappa<0,
	\end{cases}
\end{equation*}
and
\begin{equation*}
	s_\kappa(t) :=
	\begin{cases}
		\frac{1}{\sqrt{ \kappa}} \sin\left( \sqrt{\kappa} t \right), & \text{if } \kappa>0,\\
		t, & \text{if } \kappa=0,\\
		\frac{1}{\sqrt{-\kappa}} \sinh\left(\sqrt{-\kappa} t\right), & \text{if } \kappa<0.
	\end{cases}
\end{equation*}
Notice that $s_\kappa$ is increasing on $\mathbb{R}_+$ when $\kappa \leq 0$ and on $[0,D_\kappa/2]$ when $\kappa > 0$.\\
\indent Let $p_1,\dots,p_{m} \in M_\kappa^n$. For $\kappa \neq 0$, define the matrices
\begin{align*}
	\text{CM}_{\kappa}(p_1,\dots,p_{m}) &:= \big( c_\kappa(d_{ij}) \big)_{1 \leq i,j \leq m},\\
	P_{\kappa}(p_1,\dots,p_{m}) &:= \big( s_\kappa^2(d_{ij}/2) \big)_{1 \leq i,j \leq m},
\end{align*}
and their determinants
\begin{align*}
	\Delta_\kappa(p_1,\dots,p_{m}) &:= \det(\text{CM}_{\kappa}(p_1,\dots,p_{m})),\\
	\gamma_\kappa(p_1,\dots,p_{m}) &:= \det(P_{\kappa}(p_1,\dots,p_{m})).
\end{align*}
The first determinant is the Cayley-Menger determinant in spherical and hyperbolic geometry, depending on whether $\kappa>0$ or $\kappa<0$. The following is a classic fact found at the start of Section 67 (when $\kappa>0$) and in Theorem 106.1 (when $\kappa<0$) of Blumenthal's book\footnote{Blumenthal denotes the determinant $\Delta_{\kappa}(p_1,\dots,p_m)$ as $\Delta_{m}(p_1,\dots,p_{m})$ when $\kappa>0$ and $\Lambda_{m}(p_1,\dots,p_{m})$ when $\kappa<0$ in \cite{distance-geometry}.} \cite{distance-geometry}.

\begin{theorem}\label{thm:cm_determinant}
	The determinant $\Delta_\kappa(p_1,\dots,p_{m})$ vanishes for $m \geq n+2$. In contrast, $\Delta_\kappa(p_1,\dots,p_{m})$ is either 0 or has sign $\text{sgn}(\kappa)^{m+1}$ when $m \leq n+1$.
\end{theorem}

The other fact needed for the proof is the Desnanot-Jacobi identity, also known as the Lewis-Carroll identity \cite{dodgson-condensation}. See also \cite{dodgson-combinatorial} for a combinatorial proof given by Zeilberger.
\begin{theorem}\label{thm:lewis_carroll_identity}
	Let $M$ be an $n$-by-$n$ matrix. Denote by $M_{\sim i, \sim j}$ the submatrix obtained by deleting the $i$-th row and $j$-th columns of $M$. If $i<j$ and $h<k$, let $M_{\sim i,j, \sim h,k} := (M_{\sim j, \sim k})_{\sim i, \sim h}$. Then:
	\begin{equation*}
		\det(M)\det(M_{\sim 1,n, \sim 1,n}) = \det(M_{\sim 1, \sim 1}) \det(M_{\sim n, \sim n}) - \det(M_{\sim 1, \sim n}) \det(M_{\sim n, \sim 1}).
	\end{equation*}
\end{theorem}

With these tools at hand, we are ready to prove the higher dimensional generalization of Ptolemy's Theorem.
\begin{theorem}\label{thm:sign_of_ptolemys_determinant}
	Let $\kappa \neq 0$ and $p_1,\dots,p_{n+2} \in M_\kappa^{n}$. The determinant $\gamma_\kappa(p_1,\dots,p_{n+2})$ vanishes or has sign $(-1)^{n+1}$.
\end{theorem}
\begin{proof}
	Consider the matrix
	\begin{equation*}
		A :=
		\left(
		\begin{array}{c|c}
			-1 & 0 \\
			\hline
			0 & \text{CM}_\kappa(p_1,\dots,p_{n+2})
		\end{array}
		\right)
		=
		\left(
		\begin{array}{c|c}
			-1 & 0 \\
			\hline
			0 & c_\kappa(d_{ij})
		\end{array}
		\right).
	\end{equation*}
	From the half-angle formulas for trigonometric and hyperbolic functions, $c_\kappa$ and $s_\kappa$ satisfy $c_\kappa(2t)-1 = -2\kappa s_\kappa^2(t)$. If we subtract the first row of $A$ from the other rows and the first column from the others, we get the matrix
	\begin{equation*}
		B :=
		\left(
		\begin{array}{c|c}
			-1 & 1 \\
			\hline
			1 & -2 \kappa s_\kappa^2(d_{ij}/2)
		\end{array}
		\right),
	\end{equation*}
	with determinant $\det(B) = \det(A) = -\Delta_\kappa(p_1,\dots,p_{n+2})$. If all the first minors of $\text{CM}_{\kappa}(p_1,\dots,p_{n+2})$ are 0, then $\text{CM}_{\kappa}(p_1,\dots,p_{n+2})$ has rank at most $n$. As a consequence, both $A$ and $B$ have a maximum rank of $n+1$, which means that their first minors are also 0. In particular, $\gamma_\kappa(p_1,\dots,p_{n+2}) = \det(B_{\sim 1, \sim 1}) = 0$.\\
	\indent Suppose, then, that $\text{CM}_{\kappa}(p_1,\dots,p_{n+2})$ has a non-zero minor, and relabel the points so that $\Delta_{\kappa}(p_1,\dots,p_{n+1}) \neq 0$. By Theorem \ref{thm:lewis_carroll_identity}, we have
	\begin{equation*}
		\det(B)\det(B_{\sim 1,n, \sim 1,n}) = \det(B_{\sim 1, \sim 1})\det(B_{\sim n, \sim n}) - \det(B_{\sim 1, \sim n})\det(B_{\sim n, \sim 1}).
	\end{equation*}
	In other words,
	\begin{align*}
		0
		&= -\Delta_\kappa(p_1,\dots,p_{n+2})\det(B_{\sim 1,n, \sim 1,n})\\
		&= (-2\kappa)^{n+2} \ \gamma_\kappa(p_1,\dots,p_{n+2}) \ (-\Delta_\kappa(p_1,\dots,p_{n+1})) - (\det(B_{\sim 1, \sim n}))^2.
	\end{align*}
	The first line is 0 because $\Delta_\kappa(p_1,\dots,p_{n+2}) = 0$ by Theorem \ref{thm:cm_determinant}. Likewise, the sign of $\Delta_\kappa(p_1,\dots,p_{n+1})$ is $\text{sgn}(\kappa)^{n+2}$. Since $(\det(B_{\sim 1, \sim n}))^2$ is positive, the above equation forces
	\begin{equation*}
		\text{sgn}(\gamma_\kappa(p_1,\dots,p_{n+2})) = (-\text{sgn}(\kappa))^{n+2} (-\text{sgn}(\kappa)^{n+2}) = (-1)^{n+1}.
	\end{equation*}
\end{proof}

Calling Theorem \ref{thm:sign_of_ptolemys_determinant} the generalization of Ptolemy's inequality is justified for two reasons. First, the geometric results found in \cite{ptolemy-higher-dim} are analogous to the Euclidean case: Ptolemy's theorem in $\mathbb{R}^2$ holds if and only if the four points are on a line or a circle. Likewise, Theorem 4.7 of \cite{ptolemy-higher-dim} says that when $p_1, \dots, p_{n+2} \in M_{-1}^n$, $\gamma_\kappa(p_1,\dots,p_{n+2})=0$ if and only if the $n+2$ points lie either on a $M_{-1}^{n-1}$ subspace, on an $(n-1)$-dimensional sphere or limiting surface, or on a sheet of an $(n-1)$-dimensional equidistant surface. \cite{ptolemy-hyperbolic, ptolemy-spherical} have analogous results when $n=2$. Speaking of which, the second reason for using Ptolemy's name is that Theorem \ref{thm:sign_of_ptolemys_determinant} takes a familiar form when $n=2$.
\begin{corollary}[$\kappa$-Ptolemy inequality]\label{cor:ptolemy_general_curvature}
	Let $\kappa \in \mathbb{R}$ and $p_1,p_2,p_3,p_4 \in M_\kappa^2$. Then
	\begin{equation}\label{ineq:k_ptolemy}
		s_\kappa(d_{13}/2)s_\kappa(d_{24}/2) \leq s_\kappa(d_{12}/2)s_\kappa(d_{34}/2) + s_\kappa(d_{23}/2)s_\kappa(d_{41}/2).
	\end{equation}
\end{corollary}
\begin{proof}
	The case of $\kappa=0$ is the usual Ptolemaic inequality in the plane. For $\kappa \neq 0$, Theorem \ref{thm:sign_of_ptolemys_determinant} gives $\gamma_\kappa(p_1,p_2,p_3,p_4) \leq 0$. In Valentine's words, ``routine but tedious computations'' \cite{ptolemy-spherical} show that $\gamma_\kappa(p_1,p_2,p_3,p_4)$ decomposes as the product $A \cdot B \cdot C \cdot D$, where
	\begin{align*}
		A &:= -[s_\kappa(d_{13}/2)s_\kappa(d_{24}/2) + s_\kappa(d_{12}/2)s_\kappa(d_{34}/2) + s_\kappa(d_{23}/2)s_\kappa(d_{41}/2)] \\
		B &:= -\phantom{[}s_\kappa(d_{13}/2)s_\kappa(d_{24}/2) + s_\kappa(d_{12}/2)s_\kappa(d_{34}/2) + s_\kappa(d_{23}/2)s_\kappa(d_{41}/2) \\
		C &:= \phantom{-[}s_\kappa(d_{13}/2)s_\kappa(d_{24}/2) - s_\kappa(d_{12}/2)s_\kappa(d_{34}/2) + s_\kappa(d_{23}/2)s_\kappa(d_{41}/2) \\
		D &:= \phantom{-[}s_\kappa(d_{13}/2)s_\kappa(d_{24}/2) + s_\kappa(d_{12}/2)s_\kappa(d_{34}/2) - s_\kappa(d_{23}/2)s_\kappa(d_{41}/2).
	\end{align*}
	The sums $B+C, C+D, D+B$ are all non-negative, so no two of $B,C,D$ are negative. Since the product $B \cdot C \cdot D$ is also non-negative, we get that all three $B,C,D$ must be non-negative. In particular, $B \geq 0$ gives inequality (\ref{ineq:k_ptolemy}).
\end{proof}

\begin{example}
	Let $p_1 = (1,0,0)$, $p_2 = (0,1,0)$, $p_3 = (-1,0,0)$, and $p_4 = (0,-1,0)$ in $M_1^2 = \mathbb{S}^{2}$, the unit sphere in $\mathbb{R}^3$. Observe that $d_{12}=d_{23}=d_{34}=d_{41} = \pi/2$ and $d_{13}=d_{24}=\pi$. The set $X = \{p_1,p_2,p_3,p_4\}$ is a classic example of a non-Ptolemaic metric space, as
	\begin{equation*}
		d_{12}d_{34}+d_{23}d_{41} = \pi^2/2 < \pi^2 = d_{13}d_{24}.
	\end{equation*}
	However, since $X \subset M_1^2$, it satisfies the 1-Ptolemy inequality:
	\begin{equation*}
		\textstyle
		s_1\left(\frac{d_{13}}{2} \right) s_1\left(\frac{d_{24}}{2} \right)
		= 1^2
		\leq 2\left(\frac{\sqrt{2}}{2} \right)^2
		= s_1\left(\frac{d_{12}}{2} \right) s_1\left(\frac{d_{34}}{2} \right) + s_1\left(\frac{d_{23}}{2} \right) s_1\left(\frac{d_{41}}{2} \right).
	\end{equation*}
\end{example}

\begin{remark}
	Let $s_{ij} := s_\kappa(d_{ij}/2)$. D'Andrea and Sombra showed in \cite{cm-determinant-irreducible} that the determinant $\gamma_\kappa(p_1,\dots,p_n)$ is irreducible over $\mathbb{C}[s_{12}, \dots, s_{n-1,n}]$ for $n \geq 5$.
\end{remark}

\section{CAT-spaces.}
The next generalization of Ptolemy's inequality follows from a 4-point characterization of $\text{CAT}$-spaces.
\begin{prop}[\cite{non-positive-curvature}, Proposition II.1.11] \label{prop:4pt_condition_CAT_spaces}
	Let $X$ be a $\text{CAT}(\kappa)$ space. For every 4-tuple $(p_1,p_2,p_3,p_4) \in X$ such that $d_{12}+d_{23}+d_{34}+d_{41} < 2D_\kappa$, there exists a 4-tuple $(\overline{p}_1,\overline{p}_2,\overline{p}_3,\overline{p}_4) \in M_\kappa^2$ such that (addition of indices is done modulo 4):
	\begin{align*}
		d_{i,i+1} &= \overline{d}_{i,i+1} \text{ and }\\
		d_{i,i+2} &\leq \overline{d}_{i,i+2}.\\
	\end{align*}
\end{prop}

We take a moment here to bound the hyperbolicity of $\text{CAT}(\kappa)$ spaces when $\kappa<0$ using the above characterization.
\begin{corollary}
	\label{cor:hyperbolicity-cat-spaces}
	A $\text{CAT}(\kappa)$ space $X$ is hyperbolic when $\kappa<0$ and, in that case, $\operatorname{hyp}(X) \leq \ln(2)/\sqrt{-\kappa}$.
\end{corollary}
\begin{proof}
	Let $p_1,p_2,p_3,p_4 \in X$ and find the comparison points $\overline{p}_1, \overline{p}_2, \overline{p}_3, \overline{p}_4 \in M_\kappa^2$ that satisfy the 4-point characterization in Proposition \ref{prop:4pt_condition_CAT_spaces}. If $\delta > \operatorname{hyp}(M_\kappa^2)$, then
	\begin{align*}
		d_{13}+d_{24}
		&\leq \overline{d}_{13} + \overline{d}_{24} \\
		&\leq \max(\overline{d}_{12}+\overline{d}_{34}, \overline{d}_{23}+\overline{d}_{41})+2\delta \\
		&= \max(d_{12}+d_{34}, d_{23}+d_{41})+2\delta.
	\end{align*}
	Hence, $\operatorname{hyp}(X) \leq \operatorname{hyp}(M_\kappa^2)$. On the other hand, observe that $M_{-1}^2$ is isometric to the scaled space $\sqrt{-\kappa} \cdot M_{\kappa}^2$, so $\operatorname{hyp}(M_\kappa^2) = \operatorname{hyp}(M_{-1}^2)/\sqrt{-\kappa}$. Since $\operatorname{hyp}(M_{-1}^2) = \ln(2)$ (see \cite[Corollary 5.4]{strong-hyperbolicity} by Nica and {\v S}pakula), we obtain the result.
\end{proof}
\begin{remark}
	Notice that a space $X$ that is $\text{CAT}(\kappa)$ for every $\kappa \in \mathbb{R}$ is $0$-hyperbolic because the upper bound $\ln(2)/\sqrt{-\kappa}$ in Corollary \ref{cor:hyperbolicity-cat-spaces} goes to 0 as $\kappa \to -\infty$.
\end{remark}
Coming back to the main topic, the generalization of Ptolemy's inequality to CAT-spaces follows by extending a classic argument. The case when $\kappa=0$, for instance, appears after the statement of Proposition 3.1 in the paper by Buckley, Falk and Wraith \cite{ptolemy-and-cat0}.
\begin{theorem}\label{thm:ptolemy_cat_spaces}
	Let $X$ be a $\text{CAT}(\kappa)$ space and $p_1,p_2,p_3,p_4 \in X$ such that $d_{12}+d_{23}+d_{34}+d_{41} < 2D_\kappa$. Then
	\begin{equation*}\label{eq:ptolemy-cat-spaces}
		s_\kappa(d_{13}/2)s_\kappa(d_{24}/2) \leq s_\kappa(d_{12}/2)s_\kappa(d_{34}/2) + s_\kappa(d_{41}/2)s_\kappa(d_{23}/2).
	\end{equation*}
\end{theorem}
\begin{proof}
	With a set $\overline{p}_1, \overline{p}_2, \overline{p}_3, \overline{p}_4 \in M_\kappa^2$ as in Proposition \ref{prop:4pt_condition_CAT_spaces}, Corollary \ref{cor:ptolemy_general_curvature} gives
	\begin{align*}
		s_\kappa(d_{13}/2)s_\kappa(d_{24}/2)
		&\leq s_\kappa(\overline{d}_{13}/2) s_\kappa(\overline{d}_{24}/2)\\
		&\leq s_\kappa(\overline{d}_{12}/2) s_\kappa(\overline{d}_{34}/2) + s_\kappa(\overline{d}_{23}/2) s_\kappa(\overline{d}_{41}/2)\\
		&= s_\kappa(d_{12}/2) s_\kappa(d_{34}/2) + s_\kappa(d_{23}/2) s_\kappa(d_{41}/2).
	\end{align*}
\end{proof}

\begin{example}\label{ex:diameter-example}
	In $M_\kappa^2$, the conclusion of Corollary \ref{cor:ptolemy_general_curvature} is stronger than that of Theorem \ref{thm:ptolemy_cat_spaces} because the former holds without the restriction $d_{12}+d_{23}+d_{34}+d_{41} < 2D_\kappa$. However, this condition is necessary in the latter, as shown by the following example.\\
	\indent Let $\epsilon>0$. If $\kappa = (\frac{\pi}{2\pi+\epsilon})^2$, $M_\kappa^2$ is the sphere with radius $R = 2+\frac{\epsilon}{\pi}$ and diameter $D_\kappa = 2\pi+\epsilon$ (recall the definition of $M_\kappa^2$ in Section \ref{sec:definitions}). Since $\kappa < 1$, $M_\kappa^2$ is $\text{CAT}(1)$ by \cite[Theorem II.1.12]{non-positive-curvature}. Let $p_1 = (0,0,R)$, $p_2 = (R,0,0)$, $p_3 = (x_1,x_2,0)$, and $p_4 = (0,0,-R)$ be points of $M_\kappa^2$ chosen so that $d_{23} = \epsilon$. Notice that $d_{12} = d_{24} = d_{13} = d_{34} = \frac{\pi}{2\sqrt{\kappa}} = \pi + \epsilon/2$, and $d_{14} = \frac{\pi}{\sqrt{\kappa}} = 2\pi+\epsilon$. Then we have
	\begin{equation*}
		\textstyle
		d_{12}+d_{23}+d_{34}+d_{41} = 4\pi+ 3\epsilon = 4D_1 + 3\epsilon,
	\end{equation*}
	also
	\begin{align*}
		\textstyle
		s_1(d_{13}/2) s_1(d_{24}/2) = s_1(d_{12}/2) s_1(d_{34}/2) = \sin^2\left( \pi/2 + \epsilon/4 \right),
	\end{align*}
	and
	\begin{align*}
		\textstyle
		s_1(d_{41}/2) s_1(d_{23}/2) = \sin(\pi+\epsilon/2) \sin(\epsilon/2) < 0.
	\end{align*}
	Hence, the points $p_1, p_2, p_3, p_4$ do not satisfy the $1$-Ptolemy inequality even though $M_\kappa^2 \in \text{CAT}(1)$.
\end{example}

\begin{question}{1}
	The fact that $d_{12}+d_{23}+d_{34}+d_{41}=4\pi + 3\epsilon$ is more than double of $2D_1$ in Example \ref{ex:diameter-example} raises the question of whether the condition $d_{12}+d_{23}+d_{34}+d_{41} < 2D_1$ in Theorem \ref{thm:ptolemy_cat_spaces} is tight. We don't have an answer at this point.
\end{question}

\begin{question}{2}
	Analogously to the hyperbolicity of a metric space, we can define the $\kappa$-Ptolemaic defect of a metric space as:
	\begin{equation*}
		P_\kappa(X) := \sup\left[s_\kappa(d_{13}/2)s_\kappa(d_{24}/2) - \big(s_\kappa(d_{12}/2)s_\kappa(d_{34}/2) + s_\kappa(d_{41}/2)s_\kappa(d_{23}/2)\big) \right],
	\end{equation*}
	where $d_{ij}:=d_X(p_i,p_j)$ and the supremum is taken over all $\{p_1,p_2,p_3,p_4\} \subset X$. What properties does this constant have? In particular, is it stable under the Gromov-Hausdorff distance?
\end{question}

\section{The limiting case.}
\indent $\mathbb{R}$-trees are $\text{CAT}(\kappa)$ for every $\kappa \in \mathbb{R}$ \cite[Example II.1.15 (5)]{non-positive-curvature}. By Theorem \ref{thm:ptolemy_cat_spaces}, they satisfy the $\kappa$-Ptolemy inequality for all $\kappa$ and, by the Remark after Corollary \ref{cor:hyperbolicity-cat-spaces}, they also satisfy the 4-point condition. Now we ask the converse: what can we say about a space $X$ that satisfies the $\kappa$-Ptolemy inequality for all $\kappa$?\\
\indent To answer this, we apply the inverse of $s_\kappa(\cdot)$ to inequality (\ref{ineq:k_ptolemy}) and calculate the limit as $\kappa \to -\infty$ (for $\kappa<0$, this inverse is $\frac{2}{\sqrt{-\kappa}}\operatorname{arcsinh}(\cdot)$). Let $a, b, c, d > 0$ such that $c+d \leq a+b$. Define $e_\kappa(t) := \exp(\sqrt{-\kappa}t)$ for $\kappa < 0$ and $t > 0$. Notice that $\lim_{x \to \infty} \frac{\operatorname{arcsinh}(x)}{\ln(x)} = 1$, so for any function $f$ for which $\lim_{\kappa \to -\infty} f(\kappa) = \infty$, the limit $\lim_{\kappa \to -\infty} \frac{2}{\sqrt{-\kappa}} \operatorname{arcsinh}(f(\kappa))$ exists if and only if $\lim_{\kappa \to -\infty} \frac{2}{\sqrt{-\kappa}} \ln(f(\kappa))$ does, and both are equal. Then

\begin{align}
	& \lim_{\kappa \to -\infty} \textstyle \frac{2}{\sqrt{-\kappa}}\textstyle \operatorname{arcsinh}\left[\sinh\left(\frac{\sqrt{-\kappa}}{2}a\right) \cdot \sinh\left(\frac{\sqrt{-\kappa}}{2}b\right) + \sinh\left(\frac{\sqrt{-\kappa}}{2}c\right) \cdot \sinh\left(\frac{\sqrt{-\kappa}}{2}d\right)\right] \nonumber\\
	&=\lim_{\kappa \to -\infty}\textstyle \frac{2}{\sqrt{-\kappa}} \ln\left[
	\frac{e_\kappa(a/2)-e_\kappa(-a/2)}{2} \cdot
	\frac{e_\kappa(b/2)-e_\kappa(-b/2)}{2} + \frac{e_\kappa(c/2)-e_\kappa(-c/2)}{2} \cdot \frac{e_\kappa(d/2)-e_\kappa(-d/2)}{2} \right] \nonumber\\
	&= \lim_{\kappa \to -\infty}\textstyle \frac{2}{\sqrt{-\kappa}} \ln\left[
	e_\kappa\left(\frac{a+b}{2}\right) \left(\frac{1-e_\kappa(-a)}{2} \cdot
	\frac{1-e_\kappa(-b)}{2}
	+\frac{e_\kappa \left(\frac{c+d}{2} \right)}{e_\kappa \left(\frac{a+b}{2} \right)} \cdot
	\frac{1-e_\kappa(-c)}{2} \cdot
	\frac{1-e_\kappa(-d)}{2}
	\right)\right] \nonumber\\
	&= \lim_{\kappa \to -\infty}\textstyle
	\frac{2}{\sqrt{-\kappa}} \ln\left[
	\frac{1-e_\kappa(-a)}{2} \cdot
	\frac{1-e_\kappa(-b)}{2} +
	\frac{e_\kappa \left(\frac{c+d}{2} \right)}{e_\kappa \left(\frac{a+b}{2} \right)} \cdot
	\frac{1-e_\kappa(-c)}{2} \cdot
	\frac{1-e_\kappa(-d)}{2}
	\right] \label{eq:limit_1}\\
	& \phantom{\lim_{\kappa \to -\infty}} +\textstyle\frac{2}{\sqrt{-\kappa}} \ln\left[
	e_\kappa\left(\frac{a+b}{2}\right)\right]. \label{eq:limit_2}
\end{align}
Observe that, for $t>0$, $e_\kappa(-t) \to 0$ as $\kappa \to -\infty$. Then, the term inside the $\ln$ in (\ref{eq:limit_1}) converges to $1/2$ or $1/4$ depending on whether $a+b=c+d$ or not, and $(\ref{eq:limit_1})+(\ref{eq:limit_2})$ converges to $a+b$. Since $c+d \leq a+b$, we can write
\begin{equation*}
	\max(a+b, c+d) = \lim_{\kappa \to -\infty}\textstyle \frac{2}{\sqrt{-\kappa}} \operatorname{arcsinh} \left[ \sinh(\frac{\sqrt{-\kappa}}{2}a) \cdot \sinh(\frac{\sqrt{-\kappa}}{2}b) + \sinh(\frac{\sqrt{-\kappa}}{2}c) \cdot \sinh(\frac{\sqrt{-\kappa}}{2}d)\right].
\end{equation*}
Repeating the calculation above with $c=d=0$ gives
\begin{equation*}
	a+b = \lim_{\kappa \to -\infty}\textstyle \frac{2}{\sqrt{-\kappa}} \operatorname{arcsinh} \left[ \sinh(\frac{\sqrt{-\kappa}}{2}a) \cdot \sinh(\frac{\sqrt{-\kappa}}{2}b)\right].
\end{equation*}

We can now answer the question at the start of this section. If every subset of four distinct points $\{x_1, x_2, x_3, x_4\}$ of a metric space $X$ satisfies the inequality
$$
s_\kappa(d_{13}/2)s_\kappa(d_{24}/2) \leq s_\kappa(d_{12}/2)s_\kappa(d_{34}/2) + s_\kappa(d_{41}/2)s_\kappa(d_{23}/2)
$$
for all $\kappa<0$, when we apply $\frac{2}{\sqrt{-\kappa}} \operatorname{arcsinh}(\cdot)$ to both sides of the inequality, the previous calculation gives
\begin{equation}\label{ineq:infty_ptolemy}
	d_{13}+d_{24} \leq \max(d_{12}+d_{34}, d_{14}+d_{23}).
\end{equation}

Hence, we have the following characterization of $\mathbb{R}$-trees.
\begin{theorem}
	Given a metric space $(X,d_X)$, the following are equivalent:
	\begin{enumerate}
		\item\label{item:tree} $X$ can be isometrically embedded in an $\mathbb{R}$-tree.
		\item\label{item:4pt_cond} Every subset $\{x_1,x_2,x_3,x_4\} \subset X$ satisfies the 4-point condition.
		\item\label{item:k_ptolemy} Every subset $\{x_1,x_2,x_3,x_4\} \subset X$ satisfies the $\kappa$-Ptolemy inequality (\ref{ineq:k_ptolemy}) for all $\kappa$ such that $d_{12}+d_{23}+d_{34}+d_{41} < 2D_\kappa$.
	\end{enumerate}
\end{theorem}
\begin{proof}~\\
	$(\ref{item:tree} \Rightarrow \ref{item:k_ptolemy})$: Suppose $X$ is isometrically embedded in the $\mathbb{R}$-tree $T$. Since $T$ is $\text{CAT}(\kappa)$ for all $\kappa \in \mathbb{R}$, by Theorem \ref{thm:ptolemy_cat_spaces}, every 4-point subset of $T$ satisfies the $\kappa$-Ptolemy inequality for all $\kappa \in \mathbb{R}$ for which $d_{12}+d_{23}+d_{34}+d_{41} < 2D_\kappa$. In particular, the inequality holds for subsets of $X \hookrightarrow T$.\\
	$(\ref{item:k_ptolemy} \Rightarrow \ref{item:4pt_cond})$: Let $A := \{x_1, x_2, x_3, x_4\}$ be a subset of $X$ that satisfies (\ref{ineq:k_ptolemy}) for all $\kappa \in \mathbb{R}$ such that $d_{12}+d_{23}+d_{34}+d_{41} < 2D_\kappa$. If $x_i = x_j$ for some $i \neq j$, then $A$ has at most 3 points. The 4-point condition (\ref{ineq:infty_ptolemy}) follows from casework and the triangle inequality. On the other hand, if $x_i \neq x_j$ whenever $i \neq j$, then every distance $d_{ij}$ is non-zero. Hence, the calculation at the start of this section applies, we obtain (\ref{ineq:infty_ptolemy}) and, with that, item \ref{item:4pt_cond}.\\
	$(\ref{item:4pt_cond} \Rightarrow \ref{item:tree})$: This is a well-known fact; see \cite[Theorem 3.38]{R-trees}.
\end{proof}

\indent With this, our proof that the 4-point condition (\ref{ineq:infty_ptolemy}) is the tropicalization of Ptolemy's inequality is finished. In so doing, we established a connection between the 4-point condition (and more generally, Gromov hyperbolicity) with one of the great Greek theorems of antiquity. In a way, it is fascinating that stepping out of classical Euclidean geometry provided the bridge needed to link ideas such as trigonometry, curvature, group theory, trees, and genetics. Ultimately, it all started with an elegant theorem from ancient Greece, one that became central in the development of astronomy, trigonometry and mathematics as a whole.

\begin{acknowledgment}{Acknowledgment.}
	The authors wish to thank Mike Davis for helpful discussions, to Josiah Oh and Ranthony Edmonds for proofreading and suggestions to improve the manuscript, and to the anonymous reviewers for corrections and for suggesting an improved version of Example \ref{ex:diameter-example}.
\end{acknowledgment}

\vfill\eject
\end{document}

%% file: figures/Quadrilateral.tex
\begin{tikzpicture}[scale=2]
	\tikzmath{
		\pt=0.025;
		\px1 = 1.65;	\py1 = -1.3;
		\px2 = 1.25;	\py2 = 0.15;
		\px3 = 0;		\py3 = 0;
		\px4 = -0.2;	\py4 = -1;
	}
	
	\draw[fill] (\px1,\py1) circle [radius=\pt] node[below right]{$p_1$};
	\draw[fill] (\px2,\py2) circle [radius=\pt] node[above right]{$p_2$};
	\draw[fill] (\px3,\py3) circle [radius=\pt] node[above left ]{$p_3$};
	\draw[fill] (\px4,\py4) circle [radius=\pt] node[below left ]{$p_4$};
	
	\draw (\px1,\py1) -- (\px2,\py2) node[midway, right]{$d_{12}$};
	\draw (\px2,\py2) -- (\px3,\py3) node[midway, above]{$d_{23}$};
	\draw (\px3,\py3) -- (\px4,\py4) node[midway, left ]{$d_{34}$};
	\draw (\px4,\py4) -- (\px1,\py1) node[midway, below]{$d_{41}$};
	
	\draw[dashed] (\px1,\py1) -- (\px3,\py3) node[pos=0.35, above=4.5pt]{$d_{13}$};
	\draw[dashed] (\px2,\py2) -- (\px4,\py4) node[pos=0.60, above=3.0pt, left =3.0pt]{$d_{24}$};
\end{tikzpicture}

%% file: figures/Ptolemy-3D.tex
\begin{tikzpicture}[scale = 3,
	rotate around y = 115,
	rotate around x = 5,
	line join=round, line cap=round 
	]
	\tikzmath{
		\pt=0.5;
		\px1 = 0.6;		\py1 = 2;
		\px2 = 0.70;	\py2 = -0;
		\px3 = 0;		\py3 = -0;
		\px4 = -0.35;	\py4 = 0.75;
		\px1 = \px1;	\py1 = -\py1;
		\px2 = \px2;	\py2 = -\py2;
		\px3 = \px3;	\py3 = -\py3;
		\px4 = \px4;	\py4 = -\py4;
		\Ra = sqrt((\px4-\px3)*(\px4-\px3)+(\py4-\py3)*(\py4-\py3));
		\Rb = sqrt((\px2-\px3)*(\px2-\px3)+(\py2-\py3)*(\py2-\py3));
		\Rc = sqrt((\px4-\px2)*(\px4-\px2)+(\py4-\py2)*(\py4-\py2));
		\RR = sqrt(\Ra*\Ra - ((\Rc*\Rc+\Ra*\Ra-\Rb*\Rb)/(2*\Rc))*((\Rc*\Rc+\Ra*\Ra-\Rb*\Rb)/(2*\Rc)));
		\t=80;
	}
	\draw plot[mark=*, mark size=\pt] (\py1,0,\px1) node[right]{$p_1$};
	\draw plot[mark=*, mark size=\pt] (\py2,0,\px2) node[above right]{$p_2$};
	\draw plot[mark=*, mark size=\pt] (\py3,0,\px3);
	\draw plot[mark=*, mark size=\pt] (\py4,0,\px4) node[below left]{$p_4$};
	
	\draw (\py1,0,\px1) -- (\py2,0,\px2);
	\draw (\py2,0,\px2) -- (\py3,0,\px3);
	\draw (\py3,0,\px3) -- (\py4,0,\px4);
	\draw (\py4,0,\px4) -- (\py1,0,\px1);
	
	\draw[dashed] (\py3,0,\px3) arc (0:\t:\RR) node (Pf) {};
	\draw plot[mark=*, mark size=\pt] (Pf) node[above]{$p_3$};
	
	\draw[dashed] (Pf) -- (\py2,0,\px2);
	\draw[dashed] (Pf) -- (\py4,0,\px4);
	\draw[dotted] (Pf) -- (\py1,0,\px1);
	\draw[dotted] (\py3,0,\px3) -- (\py1,0,\px1);

	\draw (\py3,0,\px3) -- (\py3+0.3,0,\px3-0.15) node[inner sep = 0, above left]{$p_3'$};
\end{tikzpicture}

%% file: figures/comparison_triangle.tex
\begin{minipage}{0.48\linewidth}
	\centering
	\begin{tikzpicture}[scale=0.59, inner sep=0]
		\tikzmath{
			\pt=0.065;
			\px1=1.96;		\py1=4.18;
			\px2=0;			\py2=0;
			\px3=6;			\py3=0;
			\pl = 0.5;
			\ql = 0.5;
			\qk = 0.47;
		}
		\draw (\px1,\py1) to [out=255, in= 50] node[pos=\pl] (p0) {} (\px2,\py2);
		\draw (\px1,\py1) to [out=300, in=145] node[pos=\qk] (q0) {} (\px3,\py3);
		\draw (\px2,\py2) to [out= 15, in=165] (\px3, \py3);
		
		\draw[fill] (\px1,\py1) circle [radius=\pt] node[above=5pt]		{$x$};
		\draw[fill] (\px3,\py3) circle [radius=\pt] node[below right=5pt] {$z$};
		\draw[fill] (\px2,\py2) circle [radius=\pt] node[below  left=5pt] {$y$};
		
		\draw[fill] (p0) circle [radius=\pt] node[ left=5pt] {$p$};
		\draw[fill] (q0) circle [radius=\pt] node[right=5pt] {$q$};
		\draw (p0) -- (q0);
	\end{tikzpicture}
\end{minipage}
\hfill
\begin{minipage}{0.48\linewidth}
	\centering
	\begin{tikzpicture}[scale=0.59, inner sep=0]
	\tikzmath{
		\pt=0.065;
		\px1=1.96;		\py1=4.18;
		\px2=0;			\py2=0;
		\px3=6;			\py3=0;
		\pl = 0.5;
		\ql = 0.5;
		\qk = 0.47;
	}
	\draw (\px1,\py1) -- node[pos=\pl] (p1) {} (\px2,\py2);
	\draw (\px1,\py1) -- node[pos=\ql] (q1) {} (\px3,\py3);
	\draw (\px2,\py2) -- (\px3,\py3);
	
	\draw[fill] (\px1,\py1) circle [radius=\pt] node[above=5pt] {$\overline{x}$};
	\draw[fill] (\px2,\py2) circle [radius=\pt] node[below  left=5pt] {$\overline{y}$};
	\draw[fill] (\px3,\py3) circle [radius=\pt] node[below right=5pt] {$\overline{z}$};
	
	\draw[fill] (p1) circle [radius=\pt] node[ left=5pt] {$\overline{p}$};
	\draw[fill] (q1) circle [radius=\pt] node[right=5pt] {$\overline{q}$};
	\draw (p1) -- (q1);
	\end{tikzpicture}
\end{minipage}

%% file: figures/curvature_triangles.tex
\begin{minipage}{0.24\linewidth}
	\centering
	\begin{tikzpicture}[scale=0.40, inner sep=0]
		\tikzmath{
			\pt=0.065;
			\px1=1.96;		\py1=4.18;
			\px2=0;			\py2=0;
			\px3=6;			\py3=0;
			\pl = 0.5;
			\ql = 0.5;
			\qk = 0.47;
		}
		\draw (\px1,\py1) to [out=215, in= 80] node[pos=\pl] (p0) {} (\px2,\py2);
		\draw (\px1,\py1) to [out=330, in=115] node[pos=\qk] (q0) {} (\px3,\py3);
		\draw (\px2,\py2) to [out=345, in=195] (\px3, \py3);
		
		\draw[fill] (\px1,\py1) circle [radius=\pt] node[above=5pt]	{$x$};
		\draw[fill] (\px3,\py3) circle [radius=\pt] node[below=5pt] {$z$};
		\draw[fill] (\px2,\py2) circle [radius=\pt] node[below=5pt] {$y$};
		
		\node at (3,6.25) {$\kappa>0$};
	\end{tikzpicture}
\end{minipage}
\hfill
\begin{minipage}{0.24\linewidth}
	\centering
	\begin{tikzpicture}[scale=0.40, inner sep=0]
	\tikzmath{
		\pt=0.065;
		\px1=1.96;		\py1=4.18;
		\px2=0;			\py2=0;
		\px3=6;			\py3=0;
		\pl = 0.5;
		\ql = 0.5;
		\qk = 0.47;
	}
	\draw (\px1,\py1) -- node[pos=\pl] (p1) {} (\px2,\py2);
	\draw (\px1,\py1) -- node[pos=\ql] (q1) {} (\px3,\py3);
	\draw (\px2,\py2) -- (\px3,\py3);
	
	\draw[fill] (\px1,\py1) circle [radius=\pt] node[above=5pt] {$x$};
	\draw[fill] (\px2,\py2) circle [radius=\pt] node[below=5pt] {$y$};
	\draw[fill] (\px3,\py3) circle [radius=\pt] node[below=5pt] {$z$};		

	\node at (3,6.25) {$\kappa=0$};
	\end{tikzpicture}
\end{minipage}
\hfill
\begin{minipage}{0.24\linewidth}
	\centering
	\begin{tikzpicture}[scale=0.40, inner sep=0]
		\tikzmath{
			\pt=0.065;
			\px1=1.96;		\py1=4.18;
			\px2=0;			\py2=0;
			\px3=6;			\py3=0;
			\pl = 0.5;
			\ql = 0.5;
			\qk = 0.47;
		}
		\draw (\px1,\py1) to [out=255, in= 50] node[pos=\pl] (p0) {} (\px2,\py2);
		\draw (\px1,\py1) to [out=300, in=145] node[pos=\qk] (q0) {} (\px3,\py3);
		\draw (\px2,\py2) to [out= 15, in=165] (\px3, \py3);
		
		\draw[fill] (\px1,\py1) circle [radius=\pt] node[above=5pt]		{$x$};
		\draw[fill] (\px3,\py3) circle [radius=\pt] node[below=5pt] {$z$};
		\draw[fill] (\px2,\py2) circle [radius=\pt] node[below=5pt] {$y$};
		
		\node at (3,6.25) {$\kappa<0$};
	\end{tikzpicture}
\end{minipage}
\hfill
\begin{minipage}{0.24\linewidth}
	\centering
	\begin{tikzpicture}[scale=0.40, inner sep=0]
		\tikzmath{
			\pt=0.065;
			\px1=1.96;		\py1=4.18;
			\px2=0;			\py2=0;
			\px3=6;			\py3=0;
			\pxc=2.4;		\pyc=1.53;
			\pl = 0.5;
			\ql = 0.5;
			\qk = 0.47;
		}
		\draw (\px1,\py1) -- (\pxc,\pyc);
		\draw (\px2,\py2) -- (\pxc,\pyc);
		\draw (\px3,\py3) -- (\pxc,\pyc);
		
		\draw[fill] (\px1,\py1) circle [radius=\pt] node[above=5pt]		{$x$};
		\draw[fill] (\px3,\py3) circle [radius=\pt] node[below=5pt] {$z$};
		\draw[fill] (\px2,\py2) circle [radius=\pt] node[below=5pt] {$y$};
		\draw[fill] (\pxc,\pyc) circle [radius=\pt];	
		
		\node at (3,6.25) {$\kappa=-\infty$};
	\end{tikzpicture}
\end{minipage}